\documentclass[12pt,draft]{amsart}
\usepackage{amssymb}
\usepackage{enumitem}
\usepackage{graphicx}

\makeatletter
\@namedef{subjclassname@2020}{%
  \textup{2020} Mathematics Subject Classification}
\makeatother

\usepackage[T1]{fontenc}
\newtheorem{theorem}{Theorem}[section]

\newtheorem{lemma}[theorem]{Lemma}
\newtheorem{conjecture}[theorem]{Conjecture}

\theoremstyle{definition}

\newtheorem{remark}[theorem]{Remark}

\numberwithin{equation}{section}
\def\pmod #1{\ ({\rm{mod}}\ #1)}
\def\Z{\Bbb Z}
\def\N{\Bbb N}

\def\bg{\bigg}
\def\({\bg(}
\def\){\bg)}
\def\t{\text}

\def\gen{{\rm gen}}

\def\q{{\rm q}}

\def\ve{\varepsilon}

\theoremstyle{plain}

\begin{document}

\baselineskip=17pt
\hbox{}
\medskip

\title[Sums of four squares with a certain restriction]{Sums of four squares with a certain restriction}

\author[Y.-F. She and H.-L. Wu]{Yue-Feng She and Hai-Liang Wu }

\address {(Yue-Feng She) Department of Mathematics, Nanjing
University, Nanjing 210093, People's Republic of China}
\email{{\tt she.math@smail.nju.edu.cn}}

\address {(Hai-Liang Wu) School of Science, Nanjing University of Posts and Telecommunications, Nanjing 210023, People's Republic of China}
\email{\tt whl.math@smail.nju.edu.cn}

\date{}

\begin{abstract}
In 2016, while studying restricted sums of integral squares, Sun posed the following conjecture:
Every positive integer $n$ can be written as $x^2+y^2+z^2+w^2$ $(x,y,z,w\in\N=\{0,1,\cdots\})$ with $x+3y$ a square. Meanwhile, he also conjectured that for each positive integer $n$ there exist integers $x,y,z,w$ such that $n=x^2+y^2+z^2+w^2$ and $x+3y\in\{4^k:k\in\N\}$.
In this paper, we confirm these conjectures via some arithmetic theory of ternary quadratic forms.
\end{abstract}

\subjclass[2020]{Primary 11E25; Secondary 11E12, 11E20.}

\keywords{sums of four squares, ternary quadratic forms.}

\maketitle
\section{Introduction}
The Lagrange four-square theorem states that every positive integer can be written as the sum of four integral squares. Along this line, in 1917, Ramanujan \cite{R} claimed that there are
at most $55$ positive definite integral diagonal quaternary quadratic forms that can represent all positive integers. Later in the paper \cite{D27} Dickson confirmed that Ramanujan's claim is true for $54$ forms in Ramanujan's list and pointed out that the quaternary form $x^2+2y^2+5z^2+5w^2$ included in his list represents all positive integers except $15$.

In 2016, Sun \cite{S17a}
studied some refinements of Lagrange's theorem. For instance, he showed that
for any $k\in\{4,5,6\}$, each positive integer $n$ can be written as $x^k+y^2+z^2+w^2$ with $x,y,z,w\in\N$. In addition, let $P(X,Y,Z,W)$ be an integral polynomial. Sun called $P(X,Y,Z,W)$ a {\it suitable polynomial} if every positive integer $n$ can be written as $x^2+y^2+z^2+w^2$ $(x,y,z,w\in\N)$ with $P(x,y,z,w)$ a square. In the same paper, Sun showed that the polynomials $X,\ 2X,\ X-Y,\ 2X-2Y$ are all suitable.
Also, he showed that every positive integer $n$ can be written as $x^2+y^2+z^2+w^2$ $(x,y,z,w\in\Z)$ with $x+2y$ a square and he conjectured that $X+2Y$ is a suitable polynomial. This conjecture was later confirmed by Sun and his cooperator Y.-C. Sun in \cite{SS}. Readers may consult \cite{KS,SS,S19,WS} for the recent progress on this topic. Moreover, Sun \cite{S17a} investigated the polynomial $X+3Y$, and he \cite[Theorem 1.3(ii)]{S17a} obtained the following result:
\begin{theorem}[Sun]
Assuming the GRH {\rm (}Generalized Riemann Hypothesis{\rm )}, every positive integer $n$ can be written as $x^2+y^2+z^2+w^2$ $(x,y,z,w\in\Z)$ with $x+3y$ a square.
\end{theorem}

Based on calculations, Sun \cite[Conjecture 4.1]{S17a} posed the following conjecture.
\begin{conjecture}
$X+3Y$ is a suitable polynomial, i.e., each positive integer $n$ can be written as $x^2+y^2+z^2+w^2$ $(x,y,z,w\in\N)$ with $x+3y$ a square.
\end{conjecture}
\begin{remark}
With the help of computer, Sun \cite{Sun oeis} verified this conjecture for $n$ up to $10^{8}$. Later the authors verified this conjecture for $n$ up to $4\times10^9$.
For example,
\begin{align*}
9996&=58^2+14^2+6^2+80^2\ \text{and}\ 58+3\times14=10^2,\\
99999999&=139^2+19^2+6866^2+7269^2\ \text{and}\ 139+3\times19=14^2,\\
3999999999&=2347^2+18^2+12671^2+15295^2\ \text{and}\ 2347+3\times18=49^2.
\end{align*}
\end{remark}
In this paper, we confirm this conjecture via some arithmetic theory of ternary quadratic forms.
\begin{theorem}\label{Theorem x+3y}
Every positive integer $n$ can be written as $x^2+y^2+z^2+w^2$ $(x,y,z,w\in\N)$ with $x+3y$ a square.
\end{theorem}

If we omit the condition that $x,y,z,w\in\N$, then Sun \cite[Theorem 1.2(iii)]{S19} proved that every positive integer $n$ can be written as $x^2+y^2+z^2+w^2$ $(x,y,z,w\in\Z)$ with $x+3y\in\{4^k: k\in\N\}$ provided that any positive integer $m\equiv9\pmod{20}$ can be written as $5x^2+5y^2+z^2$ with $x,y,z\in\Z$ and $2\nmid z$. Motivated by this, Sun also posed the following conjecture.
\begin{conjecture}\label{Conjecture power of 4}
For every positive integer $n$, there exist $x,y,z,w\in\Z$ such that $n=x^2+y^2+z^2+w^2$ and $x+3y\in\{4^k: k=0,1,\cdots\}$.
\end{conjecture}
\begin{remark}
Note that there are infinitely many positive integers do not satisfy the above assertion if we add the condition that $x,y,z,w\in\N$ in this conjecture. In fact, it is known that  each integer of the form $4^{2r+1}\times2$ $(r\in\N)$ has only a single partition into four squares (cf. \cite[p. 86]{G}), i.e.,
$$2\times4^{2r+1}=(2\times4^r)^2+(2\times4^r)^2+0^2+0^2.$$
Clearly Conjecture \ref{Conjecture power of 4} does not hold for any integer of this type if we add the condition that $x,y,z,w\in\N$.
\end{remark}
In this paper, we confirm this conjecture and obtain the following result.
\begin{theorem}\label{Theorem power of 4}
Every positive integer $n$ can be written as $x^2+y^2+z^2+w^2$ $(x,y,z,w\in\Z)$ with $x+3y\in\{4^k: k=0,1,\cdots\}$. Moreover, if $4\nmid n$, then there are $x,y,z,w\in\Z$ such that $n=x^2+y^2+z^2+w^2$ and $x+3y\in\{1,4\}$.
\end{theorem}
\begin{remark}
For example,
\begin{align*}
99997&=(-98)^2+34^2+119^2+274^2\ \text{and}\ -98+3\times34=4,\\
99999&=(-29)^2+10^2+33^2+313^2\ \text{and}\ -29+3\times10=1.
\end{align*}
\end{remark}
In Section 2 we will introduce some notations and prove some lemmas which are key elements in our proofs of theorems. The proofs of our main results will be given in Section 3.

\section{Notations and some preparations}
Throughout this paper, for any prime $p$, we let $\Z_p$ denote the ring of $p$-adic integers, and let $\Z_p^{\times}$ denote the group of invertible elements in $\Z_p$. In addtion, we set
$\Z_p^{\times2}=\{x^2: x\in\Z_p^{\times}\}$ and let $M_3(\Z_p)$ denote the ring of $3\times3$ matrices with entries contained in $\Z_p$. We also adopt the standard notations of quadratic forms (readers may refer to \cite{C,Ki,Oto} for more details). Let
$$f(X,Y,Z)=aX^2+bY^2+cZ^2+rYZ+sZX+tXY$$
be an integral positive definite ternary quadratic form. Its associated matrix is
 $$M_f:=\begin{pmatrix} a & t/2 &s/2 \\t/2 & b &r/2 \\ s/2 & r/2 &c \end{pmatrix}.$$

Let $p$ be an arbitrary prime. We introduce the following notations.
$$\q(f):=\{f(x,y,z):\ (x,y,z)\in\Z^3\}.$$
$$\q_p(f):=\{f(x,y,z):\ (x,y,z)\in\Z_p^3\}.$$
In addition, let $p>2$ be a prime. We say that $f$ is unimodular over $\Z_p$ if its associated matrix $M_f\in M_3(\Z_p)$ and is invertible. We also let $\gen(f)$ denote the set of quadratic forms which are in the genus of $f$.

For any positive integer $n$, we say that $n$ can be represented by $\gen(f)$ if there exists a form $f^*\in\gen(f)$ such that $n\in\q(f^*)$. When this occurs, we write $n\in\q(\gen(f))$. By \cite[Theorem 1.3, p. 129]{C} we know that
\begin{equation}\label{Eq. local-global principle over Z}
n\in \q(\gen(f))\Leftrightarrow n\in \q_p(f)\ \text{for\ all\ primes\ $p$}.
\end{equation}
We now state our first lemma involving local representations over $\Z_2$ (cf. \cite[pp. 186--187]{Jones}).
\begin{lemma}[Jones]\label{Lemma local representation criterion of Jones}
Let $f$ be an integral positive definite ternary quadratic form, and let $n$ be a positive integer. Then
$n\in\q_2(f)$ if and only if
$$f(X,Y,Z)\equiv n\pmod{2^{r+1}}$$
is solvable, where $r$ is the $2$-adic order of $4n$.
\end{lemma}
We now give our next lemma which concerns the global representations by the form $x^2+10y^2+10z^2$.
\begin{lemma}\label{Lemma n=1,2 mod4 representation of 1,10,10}
{\rm (i)} Let $n\equiv1,2\pmod4$ be a positive integer, and let $0<\lambda\le\sqrt{10n}$ be an odd integer with $5\nmid\lambda$. Then there exist $x,y,z\in\Z$ such that
$$10n-\lambda^2=x^2+10y^2+10z^2.$$

{\rm (ii)} Let $n\equiv3\pmod4$ be a positive integer, and let $0<\delta\le\sqrt{10n}$ be an integer with $4\mid\delta$ and $5\nmid\delta$. Then there are $x,y,z\in\Z$ such that
$$10n-\delta^2=x^2+10y^2+10z^2.$$

{\rm (iii)} Let $n$ be a positive odd integer, and let $0<\mu\le\sqrt{10n}$ be an integer with $\mu\equiv2\pmod4$ and $5\nmid\mu$. Then there are $x,y,z\in\Z$ such that
$$10n-\mu^2=x^2+10y^2+10z^2.$$
\end{lemma}
\begin{proof}
(i). Let $f(X,Y,Z)=X^2+10Y^2+10Z^2$. For any prime $p\ne2,5$, it is clear that $f(X,Y,Z)$ is unimodular over $\Z_p$ and hence $\q_p(f)=\Z_p$. As $5\nmid\lambda$, we have $10n-\lambda^2\in\q_5(f)$ by the local square theorem (cf. \cite[63:1]{Oto}). When $p=2$, since $n\equiv1,2\pmod4$, we have $10n-\lambda^2\equiv1,3\pmod8$. By the local square theorem again we obtain that $10n-\lambda^2\in\Z_2^{\times2}$ if $n\equiv1\pmod4$ and that $10n-\lambda^2\in3\Z_2^{\times2}$ if $n\equiv2\pmod4$. This implies $10n-\lambda^2\in\q_2(f)$ and hence $10n-\lambda^2\in\q(\gen(f))$.

There are two classes in $\gen(f)$ and the one not containing $f$ has a representative $g(X,Y,Z)=4X^2+5Y^2+6Z^2+4ZX$. By (\ref{Eq. local-global principle over Z}) we have either $10n-\lambda^2\in\q(f)$ or $10n-\lambda^2\in\q(g)$. If $10n-\lambda^2\in\q(f)$, then we are done. Suppose now
$10n-\lambda^2\in\q(g)$, i.e., there are $x,y,z\in\Z$ such that
$$10n-\lambda^2=g(x,y,z)=4x^2+5y^2+6z^2+4zx.$$
Then clearly $2\nmid y$. As $10n-\lambda^2\equiv2n-1\equiv 5+2z^2\pmod4$, we obtain $z\equiv n+1\pmod2$. If $n\equiv1\pmod4$, then we have $10n-\lambda^2\equiv 1\equiv4x^2+5\pmod8$ and hence $x\equiv1\pmod2$. By the above we have $x-y-z\equiv0\pmod2$ in the case $n\equiv1\pmod4$. In the case $n\equiv2\pmod4$, as $z\equiv1\pmod2$, by the equality
\begin{equation}\label{Eq. automorph of g(x,y,z)}
g(X,Y,Z)=g(X+Z,\ Y,\ -Z),
\end{equation}
there must exist $x',y',z'\in\Z$ with $2\mid x'-y'-z'$ such that $10n-\lambda^2=g(x',y',z')$. One can also easily verify the following equality:
\begin{equation}\label{Eq. isometry from g to f}
f(X,Y,Z)=g\(\frac{X-Y-Z}{2},\ -Y+Z,\ Y+Z\).
\end{equation}
By this equality and the above discussion, it is easy to see that $10n-\lambda^2\in\q(f)$.

(ii). Let notations be as above. With the same reasons as in (i), we have $10n-\delta^2\in\q_p(f)$ for any prime $p\ne2$. When $p=2$, by the local square theorem we have
\begin{equation}\label{Eq. 2-adic units}
\Z_2^{\times}\subseteq\{2x^2+5y^2+5z^2: x,y,z\in\Z_2\}.
\end{equation}
Hence $5n-\delta^2/2\equiv2x^2+5y^2+5z^2\pmod8$ is solvable. This implies that the congruence equation
$10n-\delta^2\equiv f(x,y,z)\pmod{16}$ is solvable. By Lemma \ref{Lemma local representation criterion of Jones} and the above, we obtain $10n-\delta^2\in\q(\gen(f))$.

By (\ref{Eq. local-global principle over Z}) we have either $10n-\delta^2\in\q(f)$ or $10n-\delta^2\in\q(g)$. If $10n-\delta^2\in\q(f)$, then we are done. Suppose now $10n-\delta^2\in\q(g)$, i.e., there are $x,y,z$ such that $10n-\delta^2=g(x,y,z)$. Then clearly $2\mid y$. As $10n-\delta^2\equiv 2\equiv 2z^2\pmod4$, we obtain $2\nmid z$. Hence by Eq. (\ref{Eq. automorph of g(x,y,z)}) there must exist $x',y',z'\in\Z$ with $x'-y'-z'\equiv0\pmod2$ such that $10n-\delta^2=g(x',y',z')$. By Eq. (\ref{Eq. isometry from g to f}) we clearly have $10n-\delta^2\in\q(f)$.

{\rm (iii)} Let notations be as above. Clearly we have $10n-\mu^2\in\q_p(f)$ for any prime $p\ne2$. When $p=2$, by (\ref{Eq. 2-adic units}) we obtain that
the equation $5n-\mu^2/2\equiv 2x^2+5y^2+5z^2\pmod8$ is solvable. This implies that the equation
$10n-\mu^2\equiv x^2+10y^2+10z^2\pmod{16}$ is solvable. By Lemma \ref{Lemma local representation criterion of Jones} and the above, we have $10n-\mu^2\in\q(\gen(f))$.

By (\ref{Eq. local-global principle over Z}) we have either $10n-\mu^2\in\q(f)$ or $10n-\mu^2\in\q(g)$. If $10n-\mu^2\in\q(f)$, then we are done. Suppose now $10n-\mu^2\in\q(g)$, i.e., there are $x,y,z$ such that $10n-\mu^2=g(x,y,z)$. Then clearly $2\mid y$. Since $10n-\mu^2\equiv 2\equiv 2z^2\pmod4$, we get $2\nmid z$. Then by Eq. (\ref{Eq. automorph of g(x,y,z)}) there are $x',y',z'\in\Z$ with $x'-y'-z'\equiv0\pmod2$ such that $10n-\mu^2=g(x',y',z')$. By Eq. (\ref{Eq. isometry from g to f}) we clearly have $10n-\mu^2\in\q(f)$. This completes the proof.
\end{proof}
\begin{remark}
Note that
$$\{x^2+10y^2+10z^2: x,y,z\in\Z\}=\{x^2+5y^2+5z^2: x,y,z\in\Z\ \text{and}\ 2\mid y-z\}.$$
Sun have studied the latter set in \cite{S15}.
\end{remark}

Sun and his cooperator \cite[Theorem 1.1(ii)]{SS} proved that every positive integer $n$ can be written as $x^2+y^2+z^2+w^2$ $(x,y,z,w\in\N)$ with $x+2y$ a square. Motivated by this, we obtain the following stronger result.
\begin{lemma}\label{Lemma a stronger result on x+2y}
Every odd integer $n\ge8\times10^6$ can be written as $x^2+y^2+z^2+w^2$ $(x,y,z,w\in\N)$ with $x\le y$ and $x+2y$ a square.
\end{lemma}
\begin{proof}
We first note that $$8\times10^6>\bigg{[}\(\frac{2}{5^{1/4}-(4.5)^{1/4}}\)^4\bigg{]},$$ where $[\cdot]$ denotes the floor function. If $n\ge8\times10^6$, then $(5n)^{1/4}-(4.5n)^{1/4}>2$ and hence there is an integer $(4.5n)^{1/4}\le m\le (5n)^{1/4}$ such that $m\equiv\frac{n-1}{2}\pmod2$.

Now let $h(x,y,z)=x^2+5y^2+5z^2$. By \cite[pp. 112--113]{D39} we have
\begin{equation}\label{Eq. exceptional set of 1,5,5}
\q(h)=\{x\in\N: x\not\equiv\pm2\pmod5\ \text{and}\ x\ne 4^k(8l+7)\ \text{for\ any}\ k,l\in\N\}.
\end{equation}
As $5n-m^4\ge0,5n-m^4\equiv1,2\pmod4$ and $5n-m^4\not\equiv\pm2\pmod5$, we have $5n-m^4\in\q(h)$ by (\ref{Eq. exceptional set of 1,5,5}). Hence there exist $s\in\Z$ and $z,w\in\N$ such that $5n-m^4=s^2+5z^2+5w^2$. Clearly $-\sqrt{5n-m^4}\le s\le\sqrt{5n-m^4}$.
Replacing $s$ by $-s$ if necessary, we may assume that $s\in\Z$ and $s\equiv -2m^2\pmod5$. By the inequality
\begin{equation*}
s+2m^2\ge-\sqrt{5n-m^4}+2m^2=\frac{5m^4-5n}{\sqrt{5n-m^4}+2m^2}>0
\end{equation*}
we may write $s+2m^2=5y$ for some $y\in\N$. This gives
$$5n-m^4=(5y-2m^2)^2+5z^2+5w^2,$$
and hence we get
\begin{equation}\label{Eq. A in the proof of Lemma a stronger result on x+2y}
n=(m^2-2y)^2+y^2+z^2+w^2.
\end{equation}
Let $x:=m^2-2y$. Then we have
$$5x=5m^2-10y=m^2-2s\ge m^2-2\sqrt{5n-m^4}=\frac{5(m^4-4n)}{m^2+2\sqrt{5n-m^4}}>0.$$
This gives $x>0$. Moreover,
$$5(y-x)=3s+m^2\ge-3\sqrt{5n-m^4}+m^2=\frac{10(m^4-4.5n)}{m^2+3\sqrt{5n-m^4}}\ge0.$$
This gives $x\le y$.
In view of the above, we can write $n=x^2+y^2+z^2+w^2$ with $x,y,z,w\in\N$, $x\le y$ and $x+2y=m^2$. This completes the proof.
\end{proof}

We conclude this section with the following lemma.
\begin{lemma}\label{Lemma x+3y=2 square}
For every integer $n\not\equiv0\pmod4$ with $n\ge4\times10^8$, we can write $n=x^2+y^2+z^2+w^2$ $(x,y,z,w\in\N)$ with $x+3y=2m^2$ for some $m\in\N$.
\end{lemma}
\begin{proof}
We divide our proof into the following two cases.

{\bf Case 1.} $n\equiv2\pmod4$.

In this case, we write $n=2n'$ for some odd integer $n'\ge4\times10^7$. By Lemma \ref{Lemma a stronger result on x+2y} we can write $n'=x'^2+y'^2+z'^2+w'^2$ $(x',y',z',w'\in\N)$ with $x'\le y'$, $z'\le w'$ and $x'+2y'=m_0^2$ for some $m_0\in\N$. Then
$$n=2n'=(y'-x')^2+(y'+x')^2+(w'-z')^2+(z'+w')^2.$$
Letting $x:=y'-x', y:=y'+x',z:=w'-z'$ and $w:=z'+w'$, we obtain that
$$n=x^2+y^2+z^2+w^2$$
with $x+3y=(y'-x')+3(y'+x')=2(x'+2y')=2m_0^2$.

{\bf Case 2.} $n$ is odd.

We first note that $$4\times10^8>\bigg{[}\(\frac{4\sqrt{2}}{10^{1/4}-9^{1/4}}\)^4\bigg{]}.$$
If $n\ge4\times10^8$, then we have $\frac{(10n)^{1/4}}{\sqrt{2}}-\frac{(9n)^{1/4}}{\sqrt{2}}>4$ and hence there exists an integer $\frac{(9n)^{1/4}}{\sqrt{2}}\le m\le\frac{(10n)^{1/4}}{\sqrt{2}}$ such that $2\nmid m$ and $5\nmid m$. By Lemma \ref{Lemma n=1,2 mod4 representation of 1,10,10} (iii) there exist $t\in\Z$ and $z,w\in\N$ such that
$$10n-4m^4=t^2+10z^2+10w^2.$$
Clearly $-\sqrt{10n-4m^4}\le t\le\sqrt{10n-4m^4}$. Replacing $t$ by $-t$ if necessary, we may assume $t\equiv -6m^2\pmod{10}$. By the inequality
\begin{equation*}
t+6m^2\ge-\sqrt{10n-4m^4}+6m^2=\frac{10(4m^4-n)}{6m^2+\sqrt{10n-4m^4}}>0
\end{equation*}
we can write $t+6m^2=10y$ for some $y\in\N$. This implies
$$10n-4m^4=(10y-6m^2)^2+10z^2+10w^2,$$
and hence we get
$$n=(2m^2-3y)^2+y^2+z^2+w^2.$$
Let $x:=2m^2-3y$. Then we have
$$10x=2m^2-3t\ge2m^2-3\sqrt{10n-4m^4}=\frac{10(4m^4-9n)}{2m^2+3\sqrt{10n-4m^4}}\ge0.$$
This gives $x\ge0$. In view of the above, we can write $n=x^2+y^2+z^2+w^2$ $(x,y,z,w\in\N)$ with $x+3y=2m^2$.

This completes the proof of Lemma \ref{Lemma x+3y=2 square}.
\end{proof}

\section{Proofs of the main results}

{\bf Proof of Theorem \ref{Theorem x+3y}.}\
We prove our result by induction on $n$. When $n<4\times10^9$, we can verify the desired result by computer. Assume now $n\ge4\times10^9$.
We divide our remaining proof into following cases.

{\bf Case 1.} $n\equiv0\pmod4$.

If $16\mid n$, then the desired result follows immediately from the induction hypothesis. We now assume $16\nmid n$. Then we can write $n=4n'$ with $n'\not\equiv0\pmod4$. By Lemma \ref{Lemma x+3y=2 square} there exist $x_1,y_1,z_1,w_1,m_1\in\N$ such that $n'=x_1^2+y_1^2+z_1^2+w_1^2$ and $x_1+3y_1=2m_1^2$. Clearly we can write
$n=4n'=x^2+y^2+z^2+w^2$ with $x+3y=(2m_1)^2$, where $x=2x_1,y=2y_1,z=2z_1,w=2w_1\in\N$.

{\bf Case 2.} $n\equiv1,2,3\pmod4$.

We first note that $$4\times10^9>\bigg{[}\(\frac{8}{10^{1/4}-9^{1/4}}\)^4\bigg{]}>\bigg{[}\(\frac{4}{10^{1/4}-9^{1/4}}\)^4\bigg{]}.$$ If $n\ge4\times10^9$, then we have $(10n)^{1/4}-(9n)^{1/4}>8$. Hence we can find an integer $(9n)^{1/4}\le m\le (10n)^{1/4}$ satisfying the following condition:
 $$\begin{cases}2\nmid m\ \text{and}\ 5\nmid m&\mbox{if}\ n\equiv1,2\pmod4,\\4\mid m\ \text{and}\ 5\nmid m&\mbox{if}\ n\equiv3\pmod4.\end{cases}$$
By Lemma \ref{Lemma n=1,2 mod4 representation of 1,10,10} there exist $u\in\Z$ and $z,w\in\N$ such that
$$10n-m^4=u^2+10z^2+10w^2.$$
Clearly $-\sqrt{10n-m^4}\le u\le\sqrt{10n-m^4}$. Replacing $u$ by $-u$ if necessary, we may assume that $u\in\Z$ and $u\equiv -3m^2\pmod{10}$. By the inequality
\begin{equation*}
u+3m^2\ge-\sqrt{10n-m^4}+3m^2=\frac{10(m^4-n)}{\sqrt{10n-m^4}+3m^2}>0
\end{equation*}
we can write $s+3m^2=10y$ for some $y\in\N$. This gives
$$10n-m^4=(10y-3m^2)^2+10z^2+10w^2,$$
and hence
$$n=(m^2-3y)^2+y^2+z^2+w^2.$$
Let $x:=m^2-3y$. Then we have
$$10x=m^2-3u\ge m^2-3\sqrt{10n-m^4}=\frac{10(m^4-9n)}{m^2+3\sqrt{10n-m^4}}\ge0.$$
This gives $x\ge0$ and hence we have
$n=x^2+y^2+z^2+w^2$ $(x,y,z,w\in\N)$ with $x+3y=m^2$.

In view of the above, we complete the proof.\qed

{\bf Proof of Theorem \ref{Theorem power of 4}.} When $n=1,2\cdots,10$, we can verify our result by computer. Now assume $n>10$. By Lemma \ref{Lemma n=1,2 mod4 representation of 1,10,10} (i)--(ii) for any positive integer $n$ with $4\nmid n$, there are integers $a,z,w$ such that
$$10n-\varepsilon^2=a^2+10z^2+10w^2,$$
where
$$\ve=\begin{cases}1&\mbox{if}\ n\equiv1,2\pmod4,\\4&\mbox{if}\ n\equiv
3\pmod4.\end{cases}$$
Without loss of generality, we assume $a\equiv-3\ve\pmod{10}$ (otherwise we can replace $a$ by $-a$). Writing $a=10y-3\ve$, we obtain
$$10n-\ve^2=(10y-3\ve)^2+10z^2+10w^2,$$
and hence we have
$$n=(\ve-3y)^2+y^2+z^2+w^2.$$
Letting $x:=\ve-3y$, we have
\begin{equation}\label{Eq. A in the proof of Theorem power of 4}
n=x^2+y^2+z^2+w^2\ \text{and}\ x+3y=\ve.
\end{equation}
In addition, according to Lemma \ref{Lemma n=1,2 mod4 representation of 1,10,10} (iii), for any positive odd integer $n$, there are $b,z',w'\in\Z$ such that
$$10n-2^2=b^2+10z'^2+10w'^2.$$
As above, we may write $b=10y'-6$ for some $y'\in\Z$. This gives
$$10n-2^2=(10y'-6)^2+10z'^2+10w'^2,$$
and hence we have
$$n=(2-3y')^2+y'^2+z'^2+w'^2.$$
Letting $x'=2-3y'$, we obtain that
\begin{equation}\label{Eq. B in the proof of Theorem power of 4}
n=x'^2+y'^2+z'^2+w'^2\ \text{and}\ x'+3y'=2.
\end{equation}
Now let $n=2n'\equiv2\pmod4$ be a positive integer. By (\ref{Eq. exceptional set of 1,5,5}) we see that there are integers $a_1,z_1',w_1'$ such that
$$5n'-\eta^2=a_1^2+5z_1'^2+5w_1'^2,$$
where
$$\eta=\begin{cases}4&\mbox{if}\ n'\equiv 1\pmod4,\\1&\mbox{if}\ n'\equiv
3\pmod4.\end{cases}$$
Without loss of generality, we assume $a_1\equiv-2\eta\pmod5$ and then write $a_1=5y_1'-2\eta$. This gives
$$5n'-\eta^2=(5y_1'-2\eta)^2+5z_1'^2+5w_1'^2,$$
and hence
$$n'=(\eta-2y_1')^2+y_1'^2+z'^2+w'^2.$$
Letting $x_1':=\eta-2y_1'$, we obtain that $n'=x_1'^2+y_1'^2+z_1'^2+w_1'^2$ with $x_1'+2y_1'=\eta$.
As $$n=2n'=(y_1'-x_1')^2+(y_1'+x_1')^2+(z_1'-w_1')^2+(z_1'+w_1')^2,$$
letting $x_1:=y_1'-x_1', y_1:=y_1'+x_1', z_1=z_1'-w_1'$ and $w_1:=z_1'+w_1'$, we obtain that
\begin{equation}\label{Eq. C in the proof of Theorem power of 4}
n=x_1^2+y_1^2+z_1^2+w_1^2\ \text{and}\ x_1+3y_1=2\eta.
\end{equation}
By (\ref{Eq. B in the proof of Theorem power of 4}) and (\ref{Eq. C in the proof of Theorem power of 4}), we obtain that
for every positive integer $n_0$ with $4\nmid n$ there are integers $x_0,y_0,z_0,w_0$ such that
\begin{equation}\label{Eq. D in the proof of Theorem power of 4}
n_0=x_0^2+y_0^2+z_0^2+w_0^2\ \text{and}\ x_0+3y_0\in\{2\times1, 2\times2^2\}.
\end{equation}

Now we prove our result by induction on $n$. If $16\mid n$, then the desired result follows from induction hypothesis.
If $4\nmid n$, then (\ref{Eq. A in the proof of Theorem power of 4}) implies the desired result. If $n=4n'$ for some $n'\not\equiv0\pmod4$, then by (\ref{Eq. D in the proof of Theorem power of 4}) one can easily verify our result.

In view of the above, we complete the proof. \qed

\subsection*{Acknowledgements}
We thank Prof. Zhi-Wei Sun for his helpful comments. We also thank the referee for helpful comments. This research was supported by the National Natural Science
Foundation of China (grant 11971222).

\end{document}